\documentclass[12pt]{amsart}

\usepackage{amsmath,amssymb,enumerate}
\newtheorem{theorem}{Theorem}[section]
\newtheorem{claim}{}[theorem]
\newtheorem{lemma}[theorem]{Lemma}

\newtheorem{conjecture}[theorem]{Conjecture}
\theoremstyle{definition}

\newcommand{\bN}{\mathbb N}

\newcommand{\cL}{\mathcal{L}}

\newcommand{\cM}{\mathcal{M}}
\newcommand{\cP}{\mathcal{P}}

\newcommand{\cX}{\mathcal{X}}
\DeclareMathOperator{\si}{si}

\DeclareMathOperator{\cl}{cl}

\DeclareMathOperator{\PG}{PG}
\DeclareMathOperator{\EX}{EX}
\DeclareMathOperator{\GF}{GF}

\newcommand{\elem}{\varepsilon}
\newcommand{\del}{\!\setminus\!}
\newcommand{\con}{/}
\newcommand{\grf}[3]{\frac{{#1}^{#3 + #2}-1}{{#1}-1} - {#1}\left(\frac{{#1}^{2{#2}}-1}{{#1}^2-1}\right)}

\begin{document}
\sloppy

\title[Exponential growth rates]{On minor-closed classes of matroids
with exponential growth rate}

\author[Geelen]{Jim Geelen}
\address{Department of Combinatorics and Optimization,
University of Waterloo, Waterloo, Canada}
\thanks{ This research was partially supported by a grant from the
Office of Naval Research [N00014-10-1-0851].}

\author[Nelson]{Peter Nelson}

\date{\today}

\begin{abstract}
Let $\cM$ be a minor-closed class of matroids that does not
contain arbitrarily long lines.  The growth rate function, $h:\bN\rightarrow
\bN$ of $\cM$ is given by
$$h(n) = \max\left( |M|\, : \, M\in \cM, \mbox{ simple, rank-$n$}\right).$$
The Growth Rate Theorem shows that there is an
integer $c$ such that either:
$h(n)\le c\, n$, or
${n+1 \choose 2} \le h(n)\le c\, n^2$, or
there is a prime-power $q$ such that $\frac{q^n-1}{q-1} \le h(n) \le c\, q^n$;
this separates classes into those of linear density,
quadratic density, and base-$q$ exponential density.  For classes of
base-$q$ exponential density that contain no $(q^2+1)$-point line, we prove
that $h(n) =\frac{q^n-1}{q-1}$ for all sufficiently large $n$.  We also prove that,
for classes of base-$q$ exponential density that contain no
$(q^2+q+1)$-point line, there exists $k\in\bN$ such that $h(n) =
\frac{q^{n+k}-1}{q-1} - q\frac{q^{2k}-1}{q^2-1}$ for all sufficiently large $n$.
\end{abstract}

\maketitle

\section{Introduction}

We prove a refinement of the Growth Rate Theorem for certain exponentially
dense classes.  We call a class of matroids {\em minor closed} if it is closed
under both minors and isomorphism.  The \textit{growth rate function},
$h_{\cM}:\bN\rightarrow \bN$, for a class $\cM$ of matroids is defined by
\[h_{\cM}(n) = \max(|M|: M \in \cM \text{ simple, } r(M) \le n).\] 

The following striking theorem summarizes the results of several
papers [\ref{gk},\ref{gkw},\ref{gw}].
\begin{theorem}[Growth Rate Theorem]\label{growthrates}
Let $\cM$ be a minor-closed class of matroids, not containing all simple
rank-$2$ matroids. Then there is an integer $c$ such that either:
\begin{enumerate}
\item $h_{\cM}(n) \le cn$ for all $n \ge 0$, or 
\item $\binom{n+1}{2} \le h_{\cM}(n) \le cn^2$ for all $n \ge 0$, and $\cM$
contains all graphic matroids, or
\item\label{expcase} there is a prime power $q$ such that
$\frac{q^n-1}{q-1} \le h_{\cM}(n) \le cq^n$ for all $n \ge 0$, 
and $\cM$ contains all $\GF(q)$-representable matroids. 
\end{enumerate}
\end{theorem}

If $\cM$ is a minor-closed class satisfying (\ref{expcase}), then we say that
$\cM$ is \emph{base-$q$ exponentially dense}. Our main theorems
precisely determine, for many such classes, the eventual value of the growth
rate function:

\begin{theorem}\label{main2}
Let $q$ be a prime power. If $\cM$ is a base-$q$ exponentially dense
minor-closed class of matroids such that $U_{2,q^2+1} \notin \cM$, then
\[h_{\cM}(n) = \frac{q^n-1}{q-1}\] for all sufficiently large $n$. 
\end{theorem}

Consider, for example, the class $\cM$ of matroids with no
$U_{2,\ell+2}$-minor, where $\ell \ge 2$ is an integer.
By the Growth Rate Theorem, this class is base-$q$ exponentially 
dense, where $q$ is the largest prime-power not exceeding $\ell$.
Clearly $q^2>\ell$, so, by Theorem~\ref{main2},
$h_{\cM}(n) = \frac{q^n-1}{q-1}$ for all large $n$.
This special case is the main result of [\ref{gn}], which
essentially also contains a proof of Theorem~\ref{main2}. 

\begin{theorem}\label{main1}
Let $q$ be a prime power. If $\cM$ is a base-$q$
exponentially dense minor-closed class of matroids such that $U_{2,q^2+q+1}
\notin \cM$, then there is an integer $k \ge 0$ such that \[h_{\cM}(n) =
\grf{q}{k}{n}\] for all sufficiently large $n$. 
\end{theorem}

Consider, for example, any proper minor-closed subclass $\cM$ of
the GF$(q^2)$-representable
matroids that contains all GF$(q)$-representable matroids.
Such classes are all base-$q$ exponentially dense
and do not contain $U_{2,q^2+2}$, so Theorem~\ref{main1} applies;
this special case is the main result of [\ref{sqf}].

If the hypothesis of Theorem~\ref{main1} is weakened to allow $U_{2,q^2+q+1}
\in \cM$, then the conclusion no longer holds. Consider the class
$\cM_1$ defined to be the set of truncations of
all GF$(q)$-representable matroids; note that
$U_{2,q^2+q+2} \notin \cM_1$ and 
$h_{\cM_1}(n) = \frac{q^{n+1}-1}{q-1}$ for all $n \ge 2$. 

More generally, for each $k \ge 0$, if $\cM_k$ is the 
set of matroids obtained from GF$(q)$-representable matroids 
by applying $k$ truncations, then 
 $h_{\cM_k}(n) =
\frac{q^{n+k}-1}{q-1}$ for all $n \ge 2$. This expression differs from that in
Theorem~\ref{main1} by only the constant
$q\left(\tfrac{q^{2k}-1}{q^2-1}\right)$.
It is conjectured~[\ref{sqf},\ref{thesis}] that, for each $k$,
these are the extremes in a small spectrum of possible growth rate functions:
\begin{conjecture}\label{mainconj}
Let $q$ be a prime power, and
$\cM$ be a base-$q$ exponentially dense minor-closed class of matroids. There
exist integers $k$ and $d$ with $k \ge 0$ and $0 \le d \le
\frac{q^{2k}-1}{q^2-1}$, such that $h_{\cM}(n) = \frac{q^{n+k}-1}{q-1} - qd$
for all sufficiently large $n$. 
\end{conjecture}
We conjecture further that,
for every allowable $q$, $k$ and $d$, there exists a minor-closed class with
the above as its eventual growth rate function. 

There is a stronger conjecture~[\ref{thesis}] regarding the exact structure of
the extremal matroids. For a non-negative integer $k$, 
a \emph{$k$-element projection} of a matroid $M$ is a matroid of
the form $N \con C$, where $N \del C = M$, and $C$ is a $k$-element set of $N$.

\begin{conjecture}\label{bigconj}
Let $q$ be a prime power, and $\cM$ be a base-$q$ exponentially dense
minor-closed class of matroids. There exists an integer $k \ge 0$ such that, if
$M \in \cM$ is a simple matroid of sufficiently large rank with $|M| =
h_{\cM}(r(M))$, then $M$ is the simplification of a $k$-element projection of a
projective geometry over $\GF(q)$. 
\end{conjecture}

We will show, as was observed in [\ref{thesis}], that this conjecture implies the previous one; see 
Lemma~\ref{contractdisjointflat}.

\section{Preliminaries}

A matroid $M$ is called {\em $(q,k)$-full} if
$$ \varepsilon(M) \ge \frac{q^{r(M)+k}-1}{q-1} - q\frac{q^{2k}-1}{q^2-1};  $$
moreover, if strict inequality holds, $M$ is {\em $(q,k)$-overfull}.

Our proof of Theorem~\ref{main1} follows a strategy similar to
that in [\ref{sqf}]; we show that, for any integer $n>0$, 
every $(q,k)$-overfull matroid in $\EX(U_{2,q^2+q+1})$, with sufficiently large
rank, contains a $(q,k+1)$-full rank-$n$ minor.
The Growth Rate Theorem tells us that a given base-$q$
exponentially dense minor-closed
class cannot contain $(q,k)$-full matroids for arbitrarily
large $k$, so this gives the result.
The proof of Theorem~\ref{main2} is easier and will
shake out along the way.

We follow the notation of Oxley [\ref{oxley}]; flats of rank $1$, $2$ and $3$
are respectively \emph{points}, \emph{lines} and \emph{planes} of a matroid. If
$M$ is a matroid, and $X, Y \subseteq E(M)$, then $\sqcap_M(X,Y) = r_M(X) +
r_M(Y) - r_M(X \cup Y)$ is the \textit{local connectivity} between $X$ and $Y$.
If $\sqcap_M(X,Y) = 0$, then $X$ and $Y$ are \emph{skew} in $M$, and if $\cX$
is a collection of sets in $M$ such that each $X \in \cX$ is skew to the union
of the sets in $\cX - \{X\}$, then $\cX$ is a \textit{mutually skew} collection
of sets.  A pair $(F_1,F_2)$ of flats in $M$ are {\em modular}, if 
$\sqcap_M(F_1,F_2) = r_M(F_1\cap F_2)$, and a flat $F$ of $M$ is
{\em modular} if, for each flat $F'$ of $M$, the pair $(F,F')$
is modular. In a projective geometry each pair of flats is
modular and, hence, each flat is modular.

For a matroid $M$, we write $|M|$ for $|E(M)|$, and $\elem(M)$
for $|\si(M)|$, the number of points in $M$. Thus, $h_{\cM}(n) = \max(\elem(M):
M \in \cM, r(M) \le n)$. Two matroids are equal \emph{up to simplification} if
their simplifications are isomorphic. We let $\EX(M)$ denote the set of matroids
with no $M$-minor; Theorems~\ref{main2} and~\ref{main1} apply to subclasses of $\EX(U_{2,q^2+1})$ and
$\EX(U_{2,q^2+q+1})$ respectively. The following theorem of Kung [\ref{kung}] bounds
the density of a matroid in $\EX(U_{2,\ell+2})$:

\begin{theorem}\label{kungthm}
Let $\ell \ge 2$ be an integer. If $M \in \EX(U_{2,\ell+2})$ is a matroid, then
$\elem(M) \le \frac{\ell^{r(M)}-1}{\ell-1}$. 
\end{theorem}

The next result is an easy application of the Growth Rate Theorem.
\begin{lemma}\label{gkthm} 
There is a real-valued function $\alpha_{\ref{gkthm}}(n,\beta,\ell)$ so that,
for any integers $n \ge 1$ and $\ell \ge 2$, and real number $\beta > 1$, if $M
\in \EX(U_{2,\ell+2})$ is a matroid such that $\elem(M) \ge
\alpha_{\ref{gkthm}}(n,\beta,\ell) \beta^{r(M)}$, then $M$ has a
$\PG(n-1,q)$-minor for some $q > \beta$.  
\end{lemma}

The following lemma was proved in [\ref{sqf}]:
\begin{lemma}\label{skewsubset}
Let $\lambda, \mu$ be real numbers with $\lambda > 0$ and $\mu > 1$, let $t \ge
0$ and $\ell \ge 2$ be integers, and let $A$ and $B$ be disjoint sets of
elements in a matroid $M \in \EX(U_{2,\ell+2})$ with $r_M(B) \le t < r(M)$ and
$\elem_M(A) > \lambda \mu^{r_M(A)}$. Then there is a set $A' \subseteq A$ that
is skew to $B$ and satisfies $\elem_M(A') > \lambda
\left(\frac{\mu-1}{\ell}\right)^{t}\mu^{r_M(A')}$.
\end{lemma}

\section{Projections}\label{projsection}

Recall that a $k$-element projection of a matroid $M$ is a matroid of the form
$N \con C$, where $C$ is a $k$-element set of a matroid $N$
satisfying $N \del C = M$.

In this section we are concerned with projections of projective 
geometries. Consider a $k$-element set $C$ in a matroid
$N$ such that $N\del C=\PG(n+k-1,q)$ and let $M=N\con C$. Thus $M$ is a 
$k$-element projection of $\PG(n+k-1,q)$. Below are easy 
observations that we use freely.
\begin{itemize}
\item If $C$ is not independent,
then $M$ is a $(k-1)$-element projection of $\PG(n+k-1,q)$.
\item If $C$ is not coindependent, then $M$ is a $(k-1)$-element
projection of $\PG(n+k-1,q)$.
\item If $C$ is not closed in $N$, then $M$ is, up to simplification,
a $(k-1)$-element projection of $\PG(n+k-2,q)$.
\item $M$ has a $\PG(r(M)-1,q)$-restriction.
\end{itemize}

Our next result gives the density of projections of projective geometries;
given such a projection $M$, this density is determined to within a small range
by the minimum $k$ for which $M$ is a $k$-element projection. As mentioned
earlier, this theorem also tells us that Conjecture~\ref{bigconj} implies
Conjecture~\ref{mainconj}. 
\begin{lemma}\label{contractdisjointflat}
Let $q$ be a prime power, and $k \ge 0$ be an integer. If $N$ is a matroid, and
$C$ is a rank-$k$ flat of $N$ such that $N \del C \cong \PG(r(N)-1,q)$, then
$\elem(N \con C) = \elem(N \del C) - qd$ 
for some $d \in\{0,1,\ldots, \frac{q^{2k}-1}{q-1}\}$.
\end{lemma}	

\begin{proof}
Each point $P$ of $N\con C$ is a flat of
the projective geometry $N \del C$, so $|P| =
\frac{q^{r_N(P)}-1}{q-1}= 1+q \frac{q^{r_N(P)-1}-1}{q-1}$.
Therefore $\elem(N \del C) - \elem(N \con C)$ is a multiple of $q$.

Let $\cP$ denote the set of all points in $N\con C$ that
contain more than one element, and let 
$F$ be the flat of $N\del C$ spanned by the union of these points.
Choose a set $\cP_0\subseteq \cP$ of $r_{N\con C}(F)$ points
spanning $F$ in $N\con C$; if possible
choose $\cP_0$ so that it contains a set in $P\in\cP$ with $r_N(P)>2$.
Note that: (1) the points in $\cP_0$ are mutually skew in $N\con C$,
(2) each pair of flats of $N\del C$ is modular, and (3)
$C$ is a flat of $N$. It follows that
$\cP_0$ is a mutually skew collection of flats in $N$.
Now, for each $P\in \cP_0$, 
$r_N(P)> r_{N\con C}(P)$. Therefore, since
$r(N)-r(N\con C) = k$, we have
$r_{N\con C}(F) = |\cP_0|\le k$.
Moreover, if $r_{N\con C}(F) = k$,
then each set in $\cP_0$ is a line of $N\del C$,
and, hence, by our choice of $\cP_0$, each set in $\cP$ is a line in $N\del C$.

If $r_{N\con C}(F)=k$, then we have $|F| = \frac{q^{2k}-1}{q-1}$ and $|\cP|\le \frac{|F|}{q+1}$. This gives $\elem(N\del C) - \elem(N\con C) \le  q \frac{|F|}{q+1} 
=  q \frac{q^{2k}-1}{q^2-1},$ as required.
 
If $r_{N\con C}(F) <k$, then 
$\elem(N\del C) - \elem(N\con C) \le  |F|  \le  \frac{q^{2k-1}-1}{q-1}.$
It is routine to verify that
$\tfrac{q^{2k-1}-1}{q-1} < q\tfrac{q^{2k}-1}{q^2-1},$
which proves the result.
\end{proof}

The next two lemmas consider single-element projections, highlighting the importance of $U_{2,q^2+1}$ and $U_{2,q^2+q+1}$
in Theorems~\ref{main2} and~\ref{main1}.
\begin{lemma}\label{pgframe}
Let $q$ be a prime power and let $e$ be an element of a
matroid $M$ such that $M\del e\cong \PG(r(M)-1,q)$.
Then there is a unique minimal flat $F$ of $M\del e$ that
spans $e$. Moreover, if $r(M)\ge 3$ and $r_M(F)\ge 2$, then 
$M\con e$ contains a $U_{2,q^2+1}$-minor, and
if $r_M(F)\ge 3$, then 
$M\con e$ contains a $U_{2,q^2+q+1}$-minor.
\end{lemma}

\begin{proof} 
If $F_1$ and $F_2$ are two flats of $M\del e$ that span $e$,
then, since $r_M(F_1\cap F_2) + r_M(F_1\cup F_2) = r_M(F_1) + r_M(F_2)$,
it follows that $F_1\cap F_2$ also spans $e$.
Therefore there is a unique minimal flat $F$ of $M\del e$
that spans $e$.  The uniqueness of $F$ implies that 
$e$ is freely placed in $F$.

Suppose that $r_M(F)\ge 3$. 
Thus $(M\con e)|F$ is the 
truncation of a projective geometry of rank $\ge 3$.
So $M\con e$ contains a truncation of $\PG(2,q)$ as a minor;
therefore $M\con e$ has a $U_{2,q^2+q+1}$-minor.

Now suppose that $r(M)\ge 3$ and that $r_M(F)=2$.
If $F'$ is a rank-$3$ flat of $M\del e$ containing $F$,
then $\varepsilon( (M\con e)|F' ) = q^2 + 1$, so 
$M\con e$ has a $U_{2,q^2+1}$-minor.
\end{proof}

An important consequence is that, if $M$ is a simple matroid with a $\PG(r(M)-1,q)$-restriction $R$ and no $U_{2,q^2+q+1}$-minor, then every $e \in E(M) - E(R)$ is spanned by a unique line of $R$. The next result describes the structure of the projections 
in $\EX(U_{2,q^2+q+1})$.
\begin{lemma}\label{singleproj}
Let $q$ be a prime power, and $M \in \EX(U_{2,q^2+q+1})$ be a simple matroid, and $e
\in E(M)$ be such that $M \del e \cong \PG(r(M)-1,q)$.
If $L$ is the unique line of $M \del e$ that spans $e$, then  $L$ is a 
point of $M\con e$, and each line
of $M \con e$ containing $L$ has $q^2+1$ points and is modular.
\end{lemma}

\begin{proof}
Let $L'$ be a line of $M\con e$ containing $L$.
Then $L'$ is a plane of $M\del e$, so, by
Lemma~\ref{pgframe}, $L'$ has $q^2+1$ points in $M$.

Note that $e$ is freely placed on the line $L\cup\{e\}$ in $M$.
It follows that $M$ is GF$(q^2)$-representable.
Now $L'$ is a $(q^2+1)$-point line in the GF$(q^2)$-representable
matroid $M\con e$; hence,
$L'$ is modular in $M\con e$.
\end{proof}

\section{Dealing with long lines}

This section contains two lemmas that construct a $U_{2,q^2+q+1}$-minor of a
matroid $M$ with a $\PG(r(M)-1,q)$-restriction $R$ and some additional structure.
\begin{lemma}\label{longlinewin}
Let $q$ be a prime power, and $M$ be a simple matroid of rank at least $7$ such that 
\begin{itemize}
\item $M$ has a $\PG(r(M)-1,q)$-restriction $R$, and 
\item $M$ has a line $L$ containing at least $q^2+2$ points, and 
\item $E(M)\neq E(R) \cup L$,
\end{itemize}
then $M$ has a $U_{2,q^2+q+1}$-minor. 
\end{lemma}

\begin{proof}
We may assume that $E(M)=E(R)\cup E(L)\cup \{z\}$ where
$z\not\in E(M)\cup E(R)$.
Let $F$ be the minimal flat of $R$ that spans $L\cup \{z\}$.
It follows easily from 
Lemma~\ref{pgframe}, that either $M$ has a $U_{2,q^2+q+1}$-minor
or $r_M(F)\le 6$.
To simplify the proof we will relax the condition
that $r(M)\ge 7$ to $r(M)\ge 1+r_M(F)$,
and we will suppose that $(M,R)$ form a minimum rank
counterexample under these weakened hypotheses. 

Let $L_z$ denote the 
line of $R$ that spans $z$ in $M$.
Since $z\not\in L$, we have $r_M(L\cup L_z) \ge 3$.
We may assume that $r_M(L\cup L_z)= 3$, since otherwise
we could contract a point in $F-(L\cup L_z)$ to
obtain a smaller counterexample.
Simlarly, we may assume that $r_M(F)=3$ and $r(M) = 4$, 
as otherwise
we could contract an element of $F-\cl_M(L\cup L_z)$ or $E(M) - \cl_M(F)$.

By Lemma~\ref{singleproj}, $L_z$ is a point of 
$(M\con z)|R$ and each line of 
$(M\con z)|R$ is modular and has $q^2+1$ points.
One of these lines is $F$, and, since $F$ spans $L$,
$F$ spans a line with $q^2+2$ points in $M\con z$.
Let $e\in cl_{M\con z}(F)$ be an element that
is not in parallel with any element of $F$.
Since $F$ is a modular line in $(M\con z)|R$,
the point $e$ is freely placed on the line $F\cup \{e\}$
in $(M\con z)|(R\cup \{e\})$.
Therefore $\varepsilon(M\con \{e,z\}) \ge \varepsilon((M\con \{z\})|R) - q^2
= 1 + q^2(q+1) - q^2 = q^3+1$, contradicting the fact that $M\in \EX(U_{2,q^2+q+1})$.
\end{proof}

\begin{lemma}\label{longlinewin2}
Let $q$ be a prime power, and $k \ge 3$ be an integer. If $M$ is a matroid of
rank at least $k+7$, with a $PG(r(M)-1,q)$-restriction, and a
set $X \subseteq E(M)$ with $r_M(X) \le k$ and  $\epsilon(M|X) > \frac{q^{2k}-1}{q^2-1}$,
then $M$ has a $U_{2,q^2+q+1}$-minor. 
\end{lemma}

\begin{proof}
Let $R$ be a $\PG(r(M)-1,q)$-restriction of $M$. By choosing a rank-$k$ set
containing $X$, we may assume that $r_M(X) = k$. By Lemma~\ref{pgframe}, 
there is a flat $F$ of rank at most $2k$
such that $X \subseteq F$. By contracting at most $k$ points in $F - \cl_M(X)$
if this is not the case, we may assume that $r_M(F) = r_M(X)$, at the cost of
relaxing our lower bound on $r(M)$ to $r(M) \ge 7$. 

We may assume that $M$ is simple, and that $X$ is a flat of $M$, so $F
\subseteq X$. Let $n = |F| = \frac{q^k-1}{q-1}$. By Lemma~\ref{pgframe}, each
point of $X$ is spanned in $M$ by a line of $R|F$. There are
$\binom{n}{2}/\binom{q+1}{2}$ such lines, each containing $q+1$ points of $F$.
If each of these lines spans at most $(q^2-q)$ points of $X - F$, then \[|X| =
|F| + |X-F| \le \frac{q^k-1}{q-1} + \frac{(q^2-q)\binom{n}{2}}{\binom{q+1}{2}}
= \frac{q^{2k}-1}{q^2-1},\] contradicting definition of $X$. Therefore, some
line $L$ of $M|X$ contains at least $q^2+2$ points. We also have $|L| \le
q^2+q$, so a calculation gives $|X-L| > \frac{q^{2k}-1}{q^2-1} - (q^2+q) \ge
\frac{q^k-1}{q-1} = |F|$, so $X \ne F \cup L$. Applying Lemma~\ref{longlinewin}
to $M|(E(R) \cup X)$ gives the result. 
\end{proof}

\section{Matchings and unstable sets}

For an integer $k \ge 0$, a \emph{$k$-matching} of a matroid $M$ is a mutually
skew $k$-set of lines of $M$. Our first theorem was proved in [\ref{sqf}], and
also follows routinely from the much more general linear matroid matching
theorem of Lov\'asz~[\ref{lovasz}]:

\begin{theorem}\label{pgmatching} 
There is an integer-valued function $f_{\ref{pgmatching}}(q,k)$ so that, for
any prime power $q$ and integers $n \ge 1$ and $k \ge 0$, if $\cL$
is a set of lines in a matroid $M \cong \PG(n-1,q)$, then either
\begin{enumerate}[(i)] 
\item $\cL$ contains a $(k+1)$-matching of $M$, or 
\item there is a flat $F$ of $M$ with $r_M(F) \le k$, and a set $\cL_0
\subseteq \cL$ with $|\cL_0| \le f_{\ref{pgmatching}}(q,k)$, such that every
line $L \in \cL$ either intersects $F$, or is in $\cL_0$.
Moreover, if $r_M(F) = k$, then $\cL_0 = \varnothing$.  
\end{enumerate}
\end{theorem}

We now define a property in terms of a matching in a spanning projective
geometry. Let $q$ be a prime power, $M \in \EX(U_{2,q^2+q+1})$ be a simple matroid with a
$\PG(r(M)-1,q)$-restriction $R$, and $X \subseteq E(M \del R)$ be a set such
that $M|(E(R) \cup X)$ is simple. Recall that, by Lemma~\ref{pgframe}, each $x
\in X$ lies in the closure of exactly one line $L_x$ of $R$. We say that $X$ is
\emph{$R$-unstable} in $M$ if the lines $\{L_x: x \in X\}$ are a matching of
size $|X|$ in $R$. 

\begin{lemma}\label{findunstable} 
There is an integer-valued function $f_{\ref{findunstable}}(q,k)$ so that, for
any prime power $q$ and integer $k \ge 0$, if $M \in \EX(U_{2,q^2+q+1})$ is a matroid
of rank at least $3$ with a $\PG(r(M)-1,q)$-restriction $R$, then either 
\begin{enumerate}[(i)] 
\item there is an $R$-unstable set of  size $k+1$ in $M$, or

\item $R$ has a flat $F$ with rank at most $k$ such that
$\epsilon(M
\con F) \le \epsilon(R \con F) + f_{\ref{findunstable}}(q,k)$.  
\end{enumerate}
\end{lemma}

\begin{proof} 
Let $q$ be a prime power, and $k \ge 0$ be an integer. Set
$f_{\ref{findunstable}}(q,k) = (q^2+q)f_{\ref{pgmatching}}(q,k)$. Let $M$ be a
matroid with a $\PG(r(M)-1,q)$-restriction $R$. We may assume that $M$ is
simple, and that the first outcome does not hold. Let $\cL$ be the set of lines
$L$ of $R$ such that $|\cl_M(L)| > |\cl_R(L)|$. If $\cL$ contains a
$(k+1)$-matching of $R$, then choosing a point from $\cl_M(L) - \cl_R(L)$ for
each line $L$ in the matching gives an $R$-unstable set of size $k+1$. We may
therefore assume that $\cL$ contains no such matching. Thus, let $F$ and
$\cL_0$ be the sets defined in the second outcome of Theorem~\ref{pgmatching}.
Let $D = \cup_{L \in \cL_0} L$. We have $|D| \le (q^2+q)|\cL_0|
\le f_{\ref{findunstable}}(q,k)$. By Lemma~\ref{pgframe}, each point of $M \del
(E(R) \cup D)$ lies in the closure of a line in $\cL$, so is parallel to a
point of $R$ in $M \con F$. Therefore, $\elem((M \con F) \del E(R)) \le
\elem((M \con F)| D)$; the result now follows.  
\end{proof}

We use an unstable set to construct a dense minor.
Recall that $(q,k)$-full and $(q,k)$-overfull were defined 
at the start of Section 2.
\begin{lemma}\label{contractunstable} 
Let $q$ be a prime power, and $k \ge 1$ and $n > k$ be integers. If $M \in
\EX(U_{2,q^2+q+1})$ is a matroid of rank at least $n+k$ with a
$\PG(r(M)-1,q)$-restriction $R$, and $X$ is an $R$-unstable set of size $k$ in
$M$, then $M$ has a rank-$n$ $(q,k)$-full minor $N$ with a
$U_{2,q^2+1}$-restriction.  
\end{lemma}

\begin{proof} 
We may assume by taking a restriction if necessary that $r(M) = n+k$, and $E(M)
= E(R) \cup X$; we show that $N = M \con X$ has the required properties.
For each $x\in X$, let $L_x$ denote the line of $R$ that spans $X$;
thus $\{L_x\, : \, x\in X\}$ is a matching.  By the
definition of instability, it is clear that $X$ is independent, so $r(N) = n$.
Let $x \in X$, and $P$
be a plane of $R$ that contains $L_x$ and is skew to $X-\{x\}$. By
Lemma~\ref{singleproj}, $(M \con x)|P$
has a $U_{2,q^2+1}$-restriction.  Since $X-\{x\}$ is skew to $P$, $M \con X$
also has a $U_{2,q^2+1}$-restriction.

To complete the proof it is enough, by
Lemma~\ref{contractdisjointflat}, to show that $\cl_M(X)$ is disjoint
from $R$. This is trivial if $X$ is empty, so consider $x\in X$ and let
$R'=\si(R\con L_x)$.  Note that $R'\cong \PG(n+k-3,q)$ is a spanning 
restriction of $M\con L_x$ and $X-\{x\}$ is $R'$-unstable.
Inductively, we may assume that $\cl_{M\con L_x}(X-\{x\})$
is disjoint from $R\con L_x$, but this implies that $\cl_M(X)$
is disjoint from $R$, as required.
\end{proof}

\section{The spanning case}

In this section we consider matroids that are spanned by a projective geometry.
\begin{lemma}\label{spanningwin}
There is an integer-valued function $f_{\ref{spanningwin}}(n,q,k)$ such that,
for any prime power $q$ and integers $k \ge 0$ and $n > k+1$, if $M \in
\EX(U_{2,q^2+q+1})$ is a matroid of rank at least $f_{\ref{spanningwin}}(n,q,k)$ such
that 
\begin{itemize}
\item $M$ has a $\PG(r(M)-1,q)$-restriction $R$, and 
\item $M$ is $(q,k)$-overfull,
\end{itemize}
then $M$ has a rank-$n$ $(q,k+1)$-full minor $N$ 
with a $U_{2,q^2+1}$-restriction. 
\end{lemma}

\begin{proof}
Let $n \ge 2$ and $k \ge 0$ be integers, and $q$ be a prime power. Let $m >
n+k+1$ be an integer such that 
\[\grf{q}{k}{r} > \frac{q^{r+j}-1}{q-1} + \max(q^2+q,f_{\ref{pgmatching}}(q,k))\]
for all $r \ge m$ and $0 \le j < k$. We set $f_{\ref{spanningwin}}(n,q,k) = m$. 

Let $M \in \EX(U_{2,q^2+q+1})$ be a $(q,k)$-overfull matroid
of rank at least $m$,
and let $R$ be a $\PG(r(M)-1,q)$-restriction of $M$. We
will show that $M$ has the required minor $N$; we may assume that $M$ is
simple. 

\begin{claim}
If $k \ge 1$, then no line of $M$ contains more than $q^2+1$ points. 
\end{claim}

\begin{proof}[Proof of claim:]
Let $L$ be a line of $M$ containing at least $q^2+2$ points. We have $|L| \le
q^2+q$, so $|E(R) \cup L| \le \frac{q^{r(M)}-1}{q-1} + q^2+q < |M|$ by
definition of $m$. Therefore, there is a point of $M$ in neither $R$ nor $L$.
By Lemma~\ref{longlinewin}, $M$ has a $U_{2,q^2+q+1}$-minor, a contradiction. 
\end{proof}

Let $\cL$ be the set of lines of $R$, and $\cL^+$ be the set of lines of $R$
that are not lines of $M$; note that each $L \in \cL^+$ contains exactly $q+1$
points of $R$, and spans an extra point in $M$. By Lemma~\ref{pgframe}, every
point of $M \del E(R)$ is spanned by a line in $\cL^+$.

\begin{claim}
$\cL^+$ contains a $(k+1)$-matching of $R$. 
\end{claim}

\begin{proof}[Proof of claim:]
If $k = 0$, then since $|M| > |R|$, we must have $\cL^+ \ne \varnothing$, so
the claim is trivial. Thus, assume that $k \ge 1$ and that there is no such
matching. Let $F \subseteq E(R)$ and $\cL_0 \subseteq \cL$ be the sets defined
in Theorem~\ref{pgmatching}. Let $j = r_M(F)$; we know that $0 \le j \le k$,
and that $\cL_0$ is empty if $j = k$. Let $\cL_F = \{L \in \cL: |L \cap F| =
1\}$. By definition, every point of $M \del R$ is in the closure of $F$, or the
closure of a line in $\cL_F \cup \cL_0$. 

Every point of $R \del F$ lies on exactly $|F|$ lines in $\cL_F$, and each such
line contains exactly $q$ points of $R \del F$, so 
\[|\cL_F| = \frac{|F||R \del F|}{q} = \frac{(q^j-1)(q^{r(M)}-q^j)}{q(q-1)^2}.\]
Furthermore, each line in $\cL$ contains $q+1$ points of $R$, and its closure
in $M$ contains at most $q^2-q$ points of $M \del R$ by the first claim. We 
argue that $|\cl_M(F)| \le \frac{q^{2j}-1}{q^2-1}$; if $j \le 2$, then this follows from the first claim, and otherwise, we have $r(M) \ge k+7$, so the bound
follows by applying Lemma~\ref{longlinewin2} to $M$ and $\cl_M(F)$. We now estimate $|M|$. 
\begin{align*}
|M| &= |R| + |M \del E(R)| \\
& \le |R| + \sum_{L \in \cL_F \cup \cL_0}|L - E(R)| + |\cl_M(F) - F|\\
& \le \frac{q^{r(M)}-1}{q-1} + (q^2-q)(|\cL_F|+|\cL_0|) +\left(\frac{q^{2j}-1}{q^2-1} - \frac{q^j-1}{q-1}\right).
\end{align*}
Now, a calculation and our value for $\cL_F$ obtained earlier together give
$|M| \le \grf{q}{j}{r(M)} + (q^2-q)|\cL_0|.$ If $j < k$, then, since $r(M) \ge
m$ and $|\cL_0| \le f_{\ref{pgmatching}}(q,k)$, we have $|M| \le
\grf{q}{k}{r(M)}$ by definition of $m$. If $j = k$, then $|\cL_0| = 0$, so the
same inequality holds. In either case, we contradict the fact that
$M$ is $(q,k)$-overfull.
\end{proof}

Now, $\cL^+$ has a matching of size $k+1$, so by construction of $\cL^+$, there
is an $R$-unstable set $X$ of size $k+1$ in $M$. Since $r(M) \ge m > n+k+1$,
the required minor $N$ is given by  Lemma~\ref{contractunstable}. 
\end{proof}

\section{Connectivity}

A matroid $M$ is \textit{weakly round} if there is no pair of sets $A,B$ with
union $E(M)$, such that $r_M(A) \le r(M)-2$ and $r_M(B) \le r(M)-1$. Any
matroid of rank at most $2$ is clearly weakly round. Weak roundness is a very
strong connectivity notion, and is preserved by contraction; the following
lemma is easily proved, and we use it freely. 

\begin{lemma}\label{roundconnectivity}
If $M$ is a weakly round matroid, and $e \in E(M)$, then $M \con e$ is weakly round. 
\end{lemma}

The first step in our proof of the main theorems will be to reduce to the
weakly round case; the next two lemmas give this reduction.

\begin{lemma}\label{getdenserestriction}
If $M$ is a matroid, then $M$ has a weakly round restriction $N$ such that
$\elem(N) \ge \varphi^{r(N)-r(M)}\elem(M)$, where $\varphi =
\tfrac{1}{2}(1+\sqrt{5})$. 
\end{lemma}

\begin{proof}				
We may assume that $M$ is not weakly round, so $r(M) > 2$, and there are sets
$A,B$ of $M$ such that $r_M(A) = r(M)-2$, $r_M(B) = r(M)-1$, and $E(M) = A \cup
B$. Now, since $\varphi^{-1} + \varphi^{-2} = 1$, either $\elem(M|A) \ge
\varphi^{-2} \elem(M)$ or $\elem(M|B) \ge \varphi^{-1} \elem(M)$; in the first
case, by induction $M|A$ has a weakly round restriction $N$ with $\elem(N) \ge
{\varphi}^{r(N)-r(M|A)} \elem(M|A) \ge
\varphi^{r(N)-r(M)+2}\varphi^{-2}\elem(M) = \varphi^{r(N)-r(M)}\elem(M)$,
giving the result. The second case is similar. 
\end{proof}

\begin{lemma}\label{weakroundnessreduction}
Let $q$ be a prime-power, and $k\ge 0$ be an integer. If
$\cM$ is a base-$q$ exponentially dense minor-closed 
class of matroids that contains $(q,k)$-overfull matroids
of arbitrarily large rank, then
$\cM$ contains weakly round, $(q,k)$-overfull
matroids of arbitrarily large rank.
\end{lemma}

\begin{proof}
Note that $\varphi<2\le q$;
by the Growth Rate Theorem, there is an integer $t>0$
such that 
$$ \varepsilon(M) \le 
 \left(\frac{q}{\varphi}\right)^{t} \frac{q^{r(M)+k}-1}{q-1}
-q\frac{q^{2k}-1}{q^2-1},$$
for all $M\in \cM$.

For any integer $n>0$, consider
a $(q,k)$-overfull matroid $M\in \cM$
with rank at least $n+t$.
By Lemma~\ref{getdenserestriction},
$M$ has a weakly round restriction $N$ such that $\elem(N) \ge \varphi^{-s}\elem(M)$, where $s = r(M)-r(N)$. We have
\begin{eqnarray*}
\elem(N)  &\ge& \varphi^{-s} \elem(M) \\
&>& \varphi^{-s} \left(\frac{q^{r(M)+k}-1}{q-1}
-q\frac{q^{2k}-1}{q-1}\right)\\
&>& \left(\frac{q}{\varphi}\right)^{s} \frac{q^{r(N)+k}-1}{q^2-1}
-q\frac{q^{2k}-1}{q^2-1}.
\end{eqnarray*}
Thus $N$ is $(q,k)$-overfull. Moreover, by the definition of $t$,
we have $s<t$ and, hence, $r(N)> n$.
\end{proof}

%
%

\section{Exploiting connectivity}

Our next lemma exploits weak roundness by showing that any interesting low-rank 
restriction can be contracted into the span of a projective geometry.
\begin{lemma}\label{contractrestriction}
There is an integer-valued function $f_{\ref{contractrestriction}}(n,q,t,\ell)$
so that, for any prime power $q$, and integers $n \ge 1,\ell \ge 2$ and $t \ge
0$, if $M \in \EX(U_{2,\ell+2})$ is a weakly round matroid with a
$\PG(f_{\ref{contractrestriction}}(n,q,t,\ell)-1,q)$-minor, and $T$ is a
restriction of $M$ of rank at most $t$, then there is a minor $N$ of $M$ of
rank at least $n$, such that $T$ is a restriction of $N$, and $N$ 
has a $\PG(r(N)-1,q)$-restriction.
\end{lemma}

\begin{proof}
Let $n,\ell$ and $t$ be positive integers with $\ell \ge 2$. Let $n' =
\max(n,t+1)$, and set $f_{\ref{contractrestriction}}(n,q,t,\ell)$ to be an
integer $m$ such that $m \ge 2t$, and 
\[ \frac{q^m-1}{q-1} \ge \alpha_{\ref{gkthm}}(n',q-\tfrac{1}{2},\ell)\left(\frac{\ell(q-\tfrac{1}{2})}{q-\tfrac{3}{2}}\right)^t(q - \tfrac{1}{2})^m,\]
and set $f_{\ref{contractrestriction}}(n,q,t,\ell) = m$. 

Let $M \in \EX(U_{2,\ell+2})$ be a weakly round matroid with a $\PG(m-1,q)$-minor $N =
M \con C \del D$, where $r(N) = r(M) - r_M(C)$. Let $T$ be a restriction of $M$
of rank at most $t$; we show that the required minor exists. 

\begin{claim}
There is a weakly round minor $M_1$ of $M$, such that $T$ is a restriction of
$M_1$, and $M_1$ has a $\PG(n'-1,q)$-restriction $N_1$.
\end{claim}

\begin{proof}[Proof of claim:]
Let $C' \subseteq C$ be maximal such that $T$ is a restriction of $M \con C'$,
and let $M' = M \con C'$. Maximality implies that $C-C' \subseteq
\cl_{M'}(E(T))$, so $r_{M'}(C-C') \le t$. Now, $r_{M'}(E(N)) = r(N) +
r_{M'}(C-C') \le m+t$. Therefore, 
\begin{align*}
\elem_{M'}(E(N)) &= \frac{q^m-1}{q-1}\\
&\ge \alpha_{\ref{gkthm}}(n',q-\tfrac{1}{2},\ell)\ell^t(q-\tfrac{3}{2})^{-t}(q-\tfrac{1}{2})^{m+t} \\
& \ge \alpha_{\ref{gkthm}}(n',q-\tfrac{1}{2},\ell)(\ell(q-\tfrac{3}{2})^{-1})^t(q-\tfrac{1}{2})^{r_{M'}(E(N))}.
\end{align*}

By Lemma~\ref{skewsubset} applied to $E(N)$ and $E(T)$, with $\mu = q -
\tfrac{1}{2}$, there is a set $A \subseteq E(N)$, skew to $E(T)$ in $M'$, such
that
\[\elem(M'|A) \ge \alpha_{\ref{gkthm}}(n',q-\tfrac{1}{2},\ell)(q-\tfrac{1}{2})^{r(M'|A)}.\]
Therefore, Theorem~\ref{kungthm} implies that $M'|A$ has a $\PG(n'-1,q')$-minor $N_1
= (M'|A) \con C_1 \del D_1$, for some $q' > q- \tfrac{1}{2}$. Let $M_1 = M'
\con C_1$. The set $A$ is skew to $E(T)$ in $M'$, and therefore also skew to
$C-C'$, so $M'|A = (M' \con (C-C'))|A = N|A$, so $M'|A$ is
$\GF(q)$-representable, and so is its minor $N_1$. Thus, $q' = q$, and $N_1$ is
a $\PG(n'-1,q)$-restriction of $M_1$. Moreover, $C_1 \subseteq A$, so $C_1$ is
skew to $E(T)$ in $M'$, and therefore $M_1$ has $T$ as a restriction. 
The matroid $M_1$ is a contraction-minor of $M$,
so is weakly round, and thus satisfies the claim. 
\end{proof}

Let $M_2$ be a minor-minimal matroid such that:
\begin{itemize}
\item $M_2$ is a weakly round minor of $M_1$, and
\item $T$ and $N_1$ are both restrictions of $M_2$.
\end{itemize}

If $r(N_1) = r(M_2)$, then $M_2$ is the required minor of $M$. We may therefore assume that $r(M_2) > r(N_1) = n'$. We have $r(T) \le t \le n'-1 \le r(M_2)-2$, so by
weak roundness of $M_2$, there is some $e \in E(M_2)$ spanned by neither $E(T)$ nor $E(N_1)$,
contradicting minimality of $M_2$.  
\end{proof}

\section{Critical elements}	

An element $e$ in a $(q,k)$-overfull matroid $M$ is
called {\em $(q,k)$-critical} if $M\con e$ 
is not $(q,k)$-overfull.
\begin{lemma}\label{criticallines}
Let $q$ be a prime power and $k\ge 0$ be an integer.
If $e$ is a $(q,k)$-critical element in
a $(q,k)$-overfull matroid $M$, then either
\begin{enumerate}[(i)]
\item $e$ is contained in a line with at least $q^2+2$ points, or
\item $e$ is contained in $\frac{q^{2k}-1}{q^2-1} +1$ lines, each with 
at least $q+2$ points.
\end{enumerate}
\end{lemma}

\begin{proof}
Suppose otherwise.
Let $\cL$ be the set of all lines of $M$ containing $e$, and let $\cL_1$
be the set of the
$\min(|\cL|,\frac{q^{2k}-1}{q^2-1})$ longest lines in $\cL$.
Every line
in $\cL-\cL_1$ has at most $q+1$ points and every line 
in $\cL_1$ has at most $q^2+1$ points, so 
\begin{align*}
 \varepsilon(M) &\le 1 + q|\cL| + (q^2-q) |\cL_1|   \\
&\le  1+ q\varepsilon(M\con e) + (q^2-q)\frac{q^{2k}-1}{q^2-1}  \\
&\le 1+ q\left(\frac{q^{r(M)+k-1}-1}{q-1}-q\frac{q^{2k}-1}{q^2-1}\right) + (q^2-q)\frac{q^{2k}-1}{q^2-1}  \\
&= \frac{q^{r(M)+k}-1}{q-1} +q\frac{q^{2k}-1}{q^2-1},  
\end{align*}
contradicting the fact that $M$ is $(q,k)$-overfull.
\end{proof}

The following result shows that a large number of $(q,k)$-critical elements gives a denser minor.
\begin{lemma}\label{criticalwin}
There is an integer-valued function $f_{\ref{criticalwin}}(n,q,k)$ so that, for
any prime power $q$, and integers $k \ge 0$, $n > k + 1$,
if $m\ge f_{\ref{criticalwin}}(n,q,k)$ is an integer, and
$M \in \EX(U_{2,q^2+q+1})$ is a $(q,k)$-overfull,
weakly round matroid such that 
\begin{itemize}
\item $M$ has a $\PG(m-1,q)$-minor, and
\item $M$ has a rank-$m$ set of $(q,k)$-critical elements,
\end{itemize}
then $M$ has a rank-$n$, $(q,k+1)$-full minor with a $U_{2,q^2+1}$-restriction.
\end{lemma}

\begin{proof}
Let $q$ be a prime power, and $k \ge 0$ and $n \ge 2$ be integers. Let $n' =
\max(k+8,n+k+1)$, let $d = f_{\ref{findunstable}}(q,k)$, let $t = d(d+1)+k+10$,
let $s=\frac{q^{2k}-1}{q^2-1}+1$,
and set $f_{\ref{criticalwin}}(n,q,k) =
f_{\ref{contractrestriction}}(n',q,t(s+1),q^2+q-1)$. 

Let $m\ge f_{\ref{criticalwin}}(n,q,k)$ be an integer, and let
$M \in \EX(U_{2,q^2+q+1})$ be a $(q,k)$-overfull,
weakly round matroid with
a $\PG(m-1,q)$-minor and a $t$-element independent set $I$
of $(q,k)$-critical elements (note that $t \le m$).
We will show that $M$ has the required minor.

By Lemma~\ref{criticallines}, for each element $e\in I$, 
there is a set $\cL_e$ of lines containing $e$ such that either
$\cL_e$ contains a single line with at least $q^2+2$ points,
or $|\cL_e|= \frac{q^{2k}-1}{q-1}+1$ and each line in $\cL_e$
has at least $q+2$ points.
There is a restriction $K$ of $M$ with
rank at most $t(s+1)$ that contains
all the lines $(\cL_e\, : \, e\in I)$.
By Lemma~\ref{contractrestriction}, $M$ has a
minor $M_1$ of rank at least $n'$ that has a
$\PG(r(M_1)-1,q)$-restriction $R_1$, and has $K$ as 
a restriction.
By Lemma~\ref{longlinewin},
$M_1$ has at most one line containing $q^2+2$ points.

\begin{claim}
There is a $(t-9)$-element subset $I_1$ of $I$ such that,
for each $e\in I_1$, we have $r_K(\cup \cL_e )\ge k+2$.
\end{claim}

\begin{proof}[Proof of claim:]
Note that $|I| = t \ge 9$.
If $k = 0$, then every $e \in I$ satisfies the required
condition, so an arbitrary $(t-9)$-subset of $I$ will do; we may 
thus assume that $k \ge 1$.  
Since $K$ contains at most one line with at least $q^2+2$ points,
there are at most two elements $e\in I$ with $|\cL_e|=1$.
If the claim fails, there is therefore an $8$-element 
subset $I_2$ of $I$ such that
$|\cL_e|=\frac{q^{2k}-1}{q^2-1}+1$ and
$r_K(\cup \cL_e) \le k+1$ for all $e \in I_2$.

For each $e\in I_2$,
let $F_e$ be the flat of $K$ spanned by $\cup \cL_e$.
Then $(K|F_e)\con e$ has rank at most $k$ and has
more than $\frac{q^{2k}-1}{q-1}$ points. Since $k \ge 1$, 
this matroid has rank at least $2$. 
Moreover, $M_1\con e$ has rank at least $n'-1 \ge k+7$ 
and has a $\PG(r(M_1 \con e)-1,q)$-restriction,
so, by Lemma~\ref{longlinewin2}, 
$(K|F_e)\con e$ has rank $2$. Hence,
$(K|F_e)\con e$ is a line containing at least $q^2+2$ points.

There are $28$ two-element subsets of $I_2$ and,
since each set $F_e$ has rank $\le 3$, there 
are at most $24$ pairs $(e,f)$ of elements in $I_2$
such that $f\in F_e$.
Hence there is a pair $(a,b)$ of elements in $I_2$
such that $a\not\in F_b$ and $b\not\in F_a$.
Now $K\con \{a,b\}$ has two lines each containing
at least $q^2+2$ points. Moreover,
$M_1\con \{a,b\}$ has rank at least $k + 6 \ge 7$, and
has a $\PG(r(M_1 \con \{a,b\})-1,q)$-restriction,
so we obtain a contradiction to Lemma~\ref{longlinewin}.
\end{proof}

\begin{claim}
$M_1$ has an $R_1$-unstable set of size $k+1$.
\end{claim}

\begin{proof}[Proof of claim.]
Suppose otherwise. By Lemma~\ref{findunstable}, there is a flat
$F$ of $R_1$ with rank at most $k$ such that
such that $\elem(M_1 \con F) \le
\elem(R_1 \con F) + f_{\ref{findunstable}}(q,k) = \elem(R_1 \con F) + d$. Let $M_2
= M_1 \con F$; the matroid $M_2$ has a $\PG(r(M_2)-1,q)$-restriction
$R_2$, and satisfies $E(M_2) = E(R_2) \cup D$, where $|D| \le d$.

Let $I_2$ be an $M_2$-independent subset of $I_1$ of size $|I_1|-k \ge d(d+1)+1$. 
For each $e \in I_2$,  we have $r_{M_2\con e}(\cup \cL_e)\ge (k+1)-k = 1$,
so $e$ is contained in a line $L_e$ with at least $q+2$ points in $M_2$.

Each $L_e$ contains $e$, and at most one other point in $I_2$, so there are at
least $|I_2|/2 > \binom{d+1}{2}$ distinct lines $L_e$. Therefore, $M_2$ has a
collection $\cL$ of lines, each with more than $q+1$ points, such that $|\cL| >
\binom{d+1}{2}$. Each line in $\cL$ must contain a point of $M_2 \del E(R_2)$.
However, $|M_2 \del E(R_2)| \le d$, so there are at most $\binom{d}{2}$ lines
of $M$ containing two points of $M_2 \del E(R_2)$, and by Lemma~\ref{pgframe},
there are at most $d$ lines of $M$ containing $q+2$ points, but just one point
of $M_2 \del E(R_2)$. This gives $|\cL| \le d + \binom{d}{2} = \binom{d+1}{2}$,
a contradiction. 
\end{proof}
Since $r(M_1) \ge n' \ge n+k+1$, we get the required minor $N$ from the above claim and Lemma~\ref{contractunstable}. 
\end{proof}

\section{The main theorems}

The following result implies
Theorems~\ref{main2} and~\ref{main1}:

\begin{theorem}
Let $q$ be a prime power, and let $\cM \subseteq \EX(U_{2,q^2+q+1})$ be a base-$q$
exponentially dense minor-closed class of matroids. There is an integer $k \ge
0$ such that 
\[h_{\cM}(n) = \grf{q}{k}{n}\]
for all sufficiently large $n$. 
Moreover, if $\cM\subseteq \EX(U_{2,q^2+1})$, then $k=0$.
\end{theorem}		

\begin{proof}
By the Growth Rate Theorem, 
$\cM$ contains all projective geometries over $\GF(q)$ 
and, hence, $\cM$ contains $(q,0)$-full matroids of arbitrarily large
rank. We may assume that the there are $(q,0)$-overfull matroids
of arbitrarily large rank, since otherwise the theorem holds.
By the Growth Rate Theorem,
there is a maximum integer $k\ge 0$ such that
$\cM$ contains $(q,k)$-overfull matroids
of arbitrarily large rank, and there is an integer $m \ge 0$ 
such that $\PG(m-1,q') \notin \cM$ for all $q' > q$. 

To prove the result, it suffices to show that, for all $n > k+1$, 
there is a rank-$n$ matroid $M \in \cM$ that is $(q,k+1)$-full 
and has a $U_{2,q^2+1}$-restriction. Fix an integer $n > k+1$, and suppose for a contradiction that this $M$ does not exist. 

Let $s=f_{\ref{criticalwin}}(n,q,k)$, and
$m_4 = \max(m,s,f_{\ref{spanningwin}}(n,q,k))$. 
Let $m_3$ be an integer such that
\[ \frac{q^{m_3}-1}{q-1} > \alpha_{\ref{gkthm}}(m_4,q-\tfrac{1}{2},q^2+q-1)
\left(\frac{q^2+q-1}{q-\tfrac{3}{2}}\right)^s(q-\tfrac{1}{2})^{m_3+s-1}.\]
Let $m_2 = \max(m,s m_3)$, and choose
an integer $m_1>m$ such that
\[ \alpha_{\ref{gkthm}}(m_2, q-\tfrac{1}{2},q^2+q-1) (q-\tfrac{1}{2})^{r}
\le \frac{q^{r}-1}{q-1}\]
for all $r \ge m_1$. 
By Lemma~\ref{weakroundnessreduction},
$\cM$ contains weakly round, $(q,k)$-overfull matroids
of arbitrarily large rank; let $M_1\in \cM$ be a weakly round, $(q,k)$-overfull matroid
with rank at least $m_1$. By Lemma~\ref{gkthm},
$M_1$ has a $\PG(m_2,q')$ minor $N_1$
for some $q' > q - \tfrac{1}{2}$; since $m_2\ge m$, we have $q'=q$.
Let $I_1$ be an independent set of $M_1$ such that
$N_1$ is a spanning restriction of $M_1\con I_1$, and
 choose $J_1\subseteq I_1$ maximal 
such that $M_1\con J_1$ is $(q,k)$-overfull.

Let $M_2=M_1\con J_1$ and let $I_2=I_1-J_1$.
By our choice of $J_1$, each element in $I_2$
is $(q,k)$-critical in $M_2$. Since $m_2 \ge s$, 
Lemma~\ref{criticalwin} gives $|I_2|<s$. 
Choose a collection $(F_1,\ldots,F_s)$ of mutually 
skew rank-$m_3$ flats in the projective geometry $N_1$; each $F_i$ satisfies $r(M_2|F_i) \le m_3+s-1$ and
$\elem(M_2|F_i)= \frac{q^{m_3}-1}{q-1}$.
By choice of $m_3$, and by Lemma~\ref{skewsubset} with 
$\mu = q - \tfrac{1}{2}$ for each $i\in\{1,\ldots,s\}$,
there is a flat $F'_i\subseteq F_i$ of $M_2$ that 
is skew to $I_2$ and satisfies $\elem(M_2|F_i') \ge \alpha_{\ref{gkthm}}(m_4,q-\tfrac{1}{2},q^2+q-1)(q-\tfrac{1}{2})^{r_{M_2}(F_i')}$.
Note that, since the sets $(F'_1,\ldots,F'_s)$ are
mutually skew in $M_2\con I_2$ and each of these
sets is skew to $I_2$ in $M$,
the flats $(F'_1,\ldots,F'_s)$ are mutually skew in $M_2$.

By Lemma~\ref{gkthm},
$M_2|F'_i$ has a $\PG(m_4-1,q')$ minor $P_i$
for some $q' > q-\tfrac{1}{2}$; since $m_4\ge m$, we have $q'=q$.
Let $X_i$ be an independent set of $M_2|F'_2$ such that
$P_i$ is a spanning restriction of $M_2\con X_i$.
Now choose $Z\subseteq X_1\cup\cdots\cup X_s$ maximal
such that $M_2\con Z$ is $(q,k)$-overfull.
Let $M_3 = M_2\con Z$. Each element of $X_1\cup\cdots\cup X_s-Z$ 
is $(q,k)$-critical
in $M_3$, and $P_i$ is a minor of $M_3$ for each $i$. The $X_i$ 
are mutually skew in $M_3$ and hence pairwise disjoint; thus, by Lemma~\ref{criticalwin}, there exists
$i_0\in\{1,\ldots,s\}$ such that $X_{i_0}-Z=\varnothing$
and, hence, $P_{i_0}$ is a restriction of $M_3$; let $R=P_{i_0}$.

Choose a minor $M_4$ of $M_3$ that is minimal such that:
\begin{itemize}
\item $M_4$ is weakly round, and $(q,k)$-overfull,
\item $M_4$ has $R$ as a restriction.
\end{itemize}
By Lemma~\ref{spanningwin}, $r(M_4)>r(R)$.
Every element of $E(M_4)-\cl_{M_4}(E(R))$ is $(q,k)$-critical
and, since $M_4$ is weakly round,
$r(M_4 \del \cl_{M_4}(E(R)))\ge r(M_4)-2 \ge m_4-1 \ge s$.
We now get a contradiction from Lemma~\ref{criticalwin}.
\end{proof}
\section*{References}
\newcounter{refs}
\begin{list}{[\arabic{refs}]}
{\usecounter{refs}\setlength{\leftmargin}{10mm}\setlength{\itemsep}{0mm}}

\item\label{gk}
J. Geelen, K. Kabell,
Projective geometries in dense matroids, 
J. Combin. Theory Ser. B 99 (2009), 1--8.

\item\label{gkw}
J. Geelen, J.P.S. Kung, G. Whittle,
Growth rates of minor-closed classes of matroids,
J. Combin. Theory. Ser. B 99 (2009), 420--427.

\item\label{gn}
J. Geelen, P. Nelson, 
The number of points in a matroid with no n-point line as a minor, 
J. Combin. Theory. Ser. B 100 (2010), 625--630.

\item\label{gw}
J. Geelen, G. Whittle,
Cliques in dense $\GF(q)$-representable matroids, 
J. Combin. Theory. Ser. B 87 (2003), 264--269.

\item\label{kung}
J.P.S. Kung,
Extremal matroid theory, in: Graph Structure Theory (Seattle WA, 1991), 
Contemporary Mathematics, 147, American Mathematical Society, Providence RI, 1993, pp.~21--61.

\item\label{lovasz}
L. Lov\'asz,
Selecting independent lines from a family of lines in a space,
Acta Sci. Math. 42 (1980), 121--131.

\item\label{sqf}
P. Nelson,
Growth rate functions of dense classes of representable matroids, 
submitted (2011).

\item\label{thesis}
P. Nelson,
Exponentially Dense Matroids,
Ph.D thesis, University of Waterloo (2011). 

\item \label{oxley}
J. G. Oxley, 
Matroid Theory,
Oxford University Press, New York (2011).
\end{list}		
\end{document}